\documentclass{amsart}
\usepackage{amsfonts,amssymb,amsmath,amsthm}
\usepackage{url}
\usepackage{enumerate}

 \usepackage{hyperref}
\usepackage{fancyhdr}

\newtheorem{theorem}{Th\'eor\`eme}[section]
\newtheorem{lemma}[theorem]{Lemme}

\newtheorem{proposition}[theorem]{Proposition}
\newtheorem{corollary}[theorem]{Corollaire}

\theoremstyle{definition}

\newtheorem{remark}[theorem]{Remarque}
\def\limproj{\mathop{\oalign{{\rm 
lim}\cr\hidewidth$\longleftarrow$\hidewidth\cr}}}
\def\lien{\mathrel{\mkern-4mu}}
\def\too{\relbar\lien\rightarrow}
\def\tooo{\relbar\lien\relbar\lien\too}
\def\ds{\displaystyle}
\def\notdiv{\nmid}
\def\ov{\overline}
\def\div{\,\vert\,}

\def\Q{\mathbb{Q}}
\def\Z{\mathbb{Z}}

\def\No{{\rm N}}

\def\Frac#1#2{\hbox{\footnotesize $\displaystyle \frac{#1}{#2}$}}
\def\plus{\displaystyle\mathop{\raise 2.0pt \hbox{$\bigoplus 
$}}\limits}
\def\prd{ \displaystyle\mathop{\raise 2.0pt \hbox{$\prod$}}\limits}
\def\sm{  \displaystyle\mathop{\raise 2.0pt \hbox{$\sum$}}\limits}

\author[Georges  Gras]{Georges  Gras }
\address{Villa la Gardette \\ chemin Ch\^ateau Gagni\`ere \\  
F--38520 Le Bourg d'Oisans.}
\email{g.mn.gras@wanadoo.fr  {\it 
url\,:\,}\url{http://www.researchgate.net/profile/Georges_Gras}  }


\begin{document}

\title[Sur le module de Bertrandias--Payan]{Sur le module de Bertrandias--Payan \\
dans une $p$-extension -- Noyau de capitulation}

\begin{abstract} For a number field $K$ and a prime number $p$ we denote by $BP_K$ the compositum of the
cyclic $p$-extensions of $K$ embeddable in a cyclic $p$-extension of arbitrary large degree. 
Then $BP_K$ is $p$-ramified and is a finite extension of the compositum $\widetilde K$
of the $\Z_p$-extensions of $K$.
We study the transfer map $j_{L/K}$ (as a capitulation map of ideal classes) for the Bertrandias--Payan  module
${\mathcal B\mathcal P}_K := {\rm Gal}(BP_K/\widetilde K)$ in a $p$-extension $L/K$
(assuming the Leopoldt conjecture). 
In the cyclic case of degree $p$, $j_{L/K}$ is injective except if $L/K$ is kummerian,  $p$-ramified,
non globally cyclotomic but locally cyclotomic at $p$ (Theorem \ref{thmf}).
We then intend to characterize the condition $\vert {\mathcal B\mathcal P}_K\vert$ divides $\vert 
{\mathcal B\mathcal P}_L^G\vert$ ($G = {\rm Gal}(L/K)$).
So we study ${\mathcal B\mathcal P}_L^G$ when $j_{L/K}$ is not injective 
and show that it depends on the classical arithmetical properties of the Galois group (over $\widetilde K$) of
the maximal Abelian $p$-ramified pro-$p$-extension of $K$.
We give complete proofs in an elementary way using ideal approach of global class field theory.
\end{abstract}

\maketitle

\centerline{18 Octobre 2015}

\medskip
Ce texte r\'esulte d'un travail en collaboration avec Jean--Fran\c cois Jaulent (Bordeaux) et Thong Nguyen Quang Do 
(Besan\c con).  L'originalit\'e de la d\'emarche \'etant que trois techniques sont propos\'ees (une par auteur) 
pour aborder le m\^eme sujet. 
Ceci nous a sembl\'e n\'ecessaire dans la mesure o\`u toutes sont issues de la th\'eorie du corps de classes, 
mais ont \'et\'e formalis\'ees diff\'eremment (tant au plan des notions que des notations) selon les
grands axes naturels du corps de classes (classes d'id\'eaux et d'id\`eles, 
corps de classes global $p$-adique et logarithmique, cohomologie Galoisienne et th\'eorie d'Iwasawa). 
Ainsi le lecteur est-il invit\'e \`a comparer ces m\'ethodologies qui, sauf erreur, ont le m\^eme potentiel de r\'esultats, 
mais avec des avantages et inconv\'enients diff\'erents selon le but recherch\'e
(allant du plus concret et num\'erique au plus abstrait), cette multiplicit\'e \'etant souvent \`a l'origine de 
difficult\'es \'editoriales (voir un historique au \S\,\ref{histoire}).

\smallskip
 Je remercie mes deux partenaires pour les nombreux \'echanges constructifs que nous avons eus.

\section{G\'en\'eralit\'es sur le module de Bertrandias--Payan}

\subsection{Le sch\'ema de la $p$-ramification Ab\'elienne}\label{sch}
Soit $p$ un nombre premier. Pour un corps $k$ quelconque on d\'esigne par $\mu(k)$ le groupe des 
racines $p$-i\`emes de l'unit\'e de $k$. Par la commodit\'e que conf\`ere un livre, on se r\'ef\`ere le plus 
souvent \`a \cite{Gr1} (non exclusif).

\smallskip
On consid\`ere le sch\'ema ci-apr\`es (\cite{Gr1}, III.2.6.1, 
Fig. 2.2), au sens ordinaire de la notion de classes d'id\'eaux,
dans lequel $\widetilde K$ est le compos\'e des 
$\Z_p$-extensions du corps de nombres
$K$, $H_K$ le $p$-corps de classes de Hilbert de $K$, $H_K^{\rm pr}$ la 
pro-$p$-extension Ab\'elienne $p$-ramifi\'ee (i.e., non ramifi\'ee en 
dehors de $p$) maximale de $K$. 

\smallskip
On d\'esigne par $U_K:=\bigoplus_{v \div p} U_v$ le $\Z_p$-module des 
unit\'es locales $p$-principales\,\footnote{D'une mani\`ere 
g\'en\'erale, dans un groupe de nombres les \'el\'ements 
$p$-principaux sont ceux dont l'image r\'esiduelle est \'egale \`a 1 
en toute place $v \div p$. }
pour $K$, o\`u chaque $U_v$ est le groupe des unit\'es 
$v$-principales du compl\'et\'e $K_v$ de $K$ en $v \div p$, par
$W_K = {\rm tor}_{\Z_p}(U_K)  = \bigoplus_{v \div p} \mu(K_v)$, et 
par $\overline E_K$ l'adh\'erence dans $U_K$ du groupe des unit\'es 
globales ($p$-principales) $E_K$ de $K$. 

\smallskip
On pose ${\mathcal W}_K = W_K / \mu(K) \subseteq {\rm 
tor}_{\Z_p}(U_K/\overline E_K) \subseteq {\mathcal T}_K$ (inclusions valables sous la 
conjecture de Leopoldt pour $K$ et $p$~: \cite{Gr1}, Lemme III.4.2.4)~:
\unitlength=0.91cm
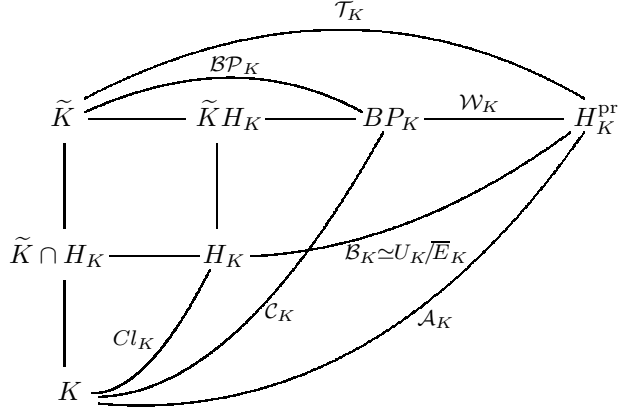
$$\vbox{\hbox{\hspace{-2.8cm}  \begin{picture}(11.5,6.2)
\put(6.4,4.50){\line(1,0){1.3}}
\put(8.7,4.50){\line(1,0){2.0}}
\put(3.85,4.50){\line(1,0){1.4}}
\put(9.2,4.65){\footnotesize$\mathcal W_K$}
\put(4.15,2.50){\line(1,0){1.3}}

\bezier{350}(3.8,4.8)(7.6,6.8)(11.0,4.8)
\put(7.4,6.0){\footnotesize${\mathcal T}_K$}

\bezier{350}(3.8,4.6)(6.0,5.6)(7.8,4.6)
\put(5.6,5.2){\footnotesize${\mathcal B\mathcal P}_K$}

\put(3.50,2.9){\line(0,1){1.25}}
\put(3.50,0.9){\line(0,1){1.25}}
\put(5.7,2.9){\line(0,1){1.25}}

\bezier{300}(3.9,0.5)(4.7,0.5)(5.6,2.3)
\put(4.2,1.2){\footnotesize$Cl_K$}

\bezier{300}(6.2,2.5)(8.5,2.6)(10.8,4.3)
\put(7.55,2.45){\footnotesize${\mathcal B}_K\!\! \simeq\!\! 
U_K\!/\!\overline E_K$}

\bezier{400}(4.0,0.35)(8.0,0.0)(11.0,4.3)
\put(8.6,1.5){\footnotesize$\mathcal A_K$}

\bezier{400}(4.0,0.45)(6.0,0.5)(8.1,4.3) 
\put(6.4,1.6){\footnotesize$\mathcal C_K$}

\put(10.85,4.4){$H_K^{\rm pr}$}
\put(5.4,4.4){$\widetilde K H_K$}
\put(7.8,4.4){$BP_K$}
\put(3.3,4.4){$\widetilde K$}
\put(5.5,2.4){$H_K$}
\put(2.7,2.4){$\widetilde K \cap H_K$}
\put(3.4,0.40){$K$}
\end{picture}   }} $$
\unitlength=1.0cm

\subsection{Historique de la $p$-ramification}\label{histoire}
Les $p$-groupes de torsion $\mathcal T$ jouent un r\^ole capital dans 
toutes les questions de type {\it Th\'eorie du Corps de Classes} et plus pr\'ecisement {\it 
Th\'eorie de la $p$-Ramification Ab\'elienne}, et ont \'et\'e longuement \'etudi\'es, sous la conjecture de Leopoldt, selon la 
chronologie approximative suivante~:  

\smallskip\noindent
\cite{Kub} (1957), \cite{Mi} (1978), \cite{Gr2} (1982), \cite{Gr3} (1983), \cite{Ja1} (1983), \cite{Ja2} (1984),
\cite{Gr4} (1985), \cite{Gr5} (1986), \cite{Ng1} (1986), \cite{MoNg1} (1987), \cite{He} (1988), \cite{Ja3} (1990), \cite{Th} (1993), \cite{Ja4} (1998), \cite{Seo1} (2011), \cite{Seo2} (2013), \cite{Seo3} (2013), \cite{Gr6} (2014), \cite{MoNg2} (2015), \cite{Gr7} (2015), 
\cite{Seo4} (2015),  \cite{Seo5} (2015), \cite{Seo6} (2015).

\smallskip
Les sujets essentiels d\'evelopp\'es dans ce cadre sont les suivants~:

\smallskip
(i) d\'etermination du radical initial dans les $\Z_p$-extensions lorsque $\mu(K) \ne 1$, \`a savoir, trouver les nombres 
$\alpha \in K^\times$ tels que $K(\sqrt[p] \alpha) \subset \widetilde K$ (\cite{CK}, \cite{Gr4}, \cite{He}, \cite{MoNg2}, \cite{Th}, \cite{JaM1}, \cite{Seo1}, \cite{Seo2}), 

\smallskip
(ii) \'etude des notions de $p$-rationalit\'e et $p$-r\'egularit\'e, la $p$-rationalit\'e de $K$ \'etant la nullit\'e de ${\mathcal T}_K$ 
sous la conjecture de Leopoldt, et la $p$-r\'egularit\'e celle du $p$-Sylow ${\rm R}_2(K)$ du noyau  (dans ${\rm K}_2(K)$) des symboles r\'eguliers (\cite{Gr1}, \cite{Gr5}, \cite{GrJ},\cite{Ja4}, \cite{JaNg}, \cite{Mo}, \cite{MoNg1}), ainsi que des cas particuliers de ces notions, notamment pour $p=2$ (J-F. Jaulent, F. Soriano--Gafiuk, O. Sauzet). 

\smallskip
(iii) g\'en\'eralisations de la notion de $p$-tours de corps de classes et structure du groupe de Galois de la
pro-$p$-extension maximale d'un corps de nombres (\cite{JaS}, \cite{JaM1}, \cite{JaM2}, \cite{Mai1} \cite{Mai2}, \cite{Seo6}),

\smallskip
(iv) obtention de r\'esultats num\'eriques sur le radical initial, le groupe ${\mathcal T}_K$, et sur la structure de
${\rm Gal}(K^{\rm ab}/ K)$ o\`u $K^{\rm ab}$ est la pro-$p$-extension Ab\'elienne maximale de $K$ 
(\cite{Cha}, \cite{He}, \cite{PV}, \cite{Th}, \cite{AS}, \cite{Gr6}), 

\smallskip
(v) analyse d'aspects conjecturaux lorsque $p \to \infty$ (e.g. \cite{Gr7} conjecturant la $p$-ratio\-nalit\'e de tout corps de 
nombres pour $p$ assez grand). R\'ecemment, la notion de $p$-rationalit\'e a acquis une importance nouvelle dans le cadre des repr\'esentations Galoisiennes en direction d'autres conjectures (e.g. \cite{Gre}).

\smallskip
Les techniques utilis\'ees \'etant de nature ``classes d'id\'eaux'' ou bien de nature ``Diviseurs'', aux sens g\'en\'eralis\'es des 
infinit\'esimaux d\'evelopp\'es au d\'ebut par G. Gras et  J-F. Jaulent et leurs \'el\`eves, ou encore de nature ``Cohomologie 
Galoisienne  et Th\'eorie d'Iwasawa'' d\'evelopp\'ee par T. Nguyen Quang Do et ses \'el\`eves, puis repris plus r\'ecemment au moyen 
d'autres points de vue de type cohomologique, comme par exemple par Seo, souvent ind\'ependamment des travaux pr\'ec\'edents, lesquels ont \'et\'e pour la plupart regroup\'es en d\'etail dans \cite{Gr1} 
(2003) et notamment dans l'Edition Springer 2005 corrig\'ee et augment\'ee.

\smallskip
Dans \cite{Ja6} et \cite{Ng2}, ainsi que dans l'article pr\'esent, on trouvera des compl\'ements et mises au point \`a ce sujet au moyen de 
l'\'etude du module de Bertrandias--Payan (cf. \S\,\ref{mbp}) qui se pr\^ete bien \`a l'utilisation de ces techniques, les trois auteurs ayant le projet de publier un document explicatif en anglais sur ces trois approches similaires, mais p\^atissant jusqu'ici de 
l'absence de ``table de concordance'' d\'etaill\'ee et de plusieurs textes publi\'es en fran\c cais, ce qui est sans doute \`a l'origine 
des difficult\'es d'analyse et de comparaison des r\'esultats de la litt\'erature et bien souvent de l'ignorance des r\'esultats ant\'erieurs.

\subsection{Le module de Bertrandias--Payan}\label{mbp}
 On sait que sous la conjecture de Leopoldt pour $K$ et $p\ne 2$, ${\mathcal W_K} \subseteq  {\rm tor}_{\Z_p}(U_K/\overline E_K)$ 
fixe l'exten\-sion $BP_K$ qui est la pro-$p$-extension maximale de $K$ compos\'ee des 
$p$-extensions cycliques plongeables dans une $p$-extension cyclique de $K$ de degr\'e arbitrairement 
grand (\cite{Gr1}, Corollary III.4.15.8)~; 
le groupe ${\mathcal B\mathcal P}_K := {\rm Gal}(BP_K/\widetilde K) \simeq {\mathcal T}_K/{\mathcal W}_K$, 
qui est le groupe de torsion du groupe ${\mathcal C}_K= {\rm Gal}(BP_K/K)$, est appel\'e depuis \cite{Ng1} le module 
de Bertrandias--Payan \`a la suite de l'article \cite{BP}.

\smallskip
Si l'on se r\'ef\`ere \`a \cite{Ja2}, \cite{Ng1}, ou \`a \cite{Gr1}, III.4.15.3, 4.15.4, $BP_K$ est aussi le compos\'e 
des $p$-extensions cycliques de $K$ localement plongeables (en toute place $v$ finie) dans une $\Z_p$-extension de $K_v$.

\smallskip
De nombreux travaux, comme ceux mentionn\'es dans l'historique
pr\'ec\'edent, \'evoquent le r\^ole de ${\mathcal B\mathcal P}$ dans 
la th\'eorie du corps de classes, aussi nous souhaitons montrer ici 
que l'on dispose depuis longtemps des techniques permettant 
d'\'etudier les propri\'et\'es de  ce module. 

\smallskip
Soit $S$ un ensemble fini de places de $K$ contenant 
l'ensemble $S_p:= \{v,\ v\div p\}$ des $p$-places~; $S$ sera choisi en fonction de
l'extension $L/K$ consid\'er\'ee et en particulier contiendra les 
places ramifi\'ees dans $L/K$. Les r\'esultats d\'ependront essentiellement de $S_p$. 
Notre ensemble $S$ jouera un r\^ole diff\'erent de celui utilis\'e dans  \cite{Ja6} et \cite{Ng2}.

\smallskip
On pose ${\mathcal K}_1^\times=K_{1}^\times \otimes \Z_p$, o\`u 
$K_{1}^\times$ est le groupe des \'el\'ements $p$-principaux de 
$K^\times$ \'etrangers \`a $S$ puis ${\mathcal I}_K=I_{K} \otimes \Z_p$, o\`u $I_{K}$ est le 
groupe des id\'eaux de $K$  \'etrangers \`a $S$, et ${\mathcal P}_K = \{(x), \ x \in {\mathcal K}_1^\times\} = 
P_{K} \otimes \Z_p$, o\`u $P_{K}$ est le sous-groupe des 
id\'eaux principaux \'etrangers \`a $S$. On pose ${\mathcal E}_K = E_K\otimes \Z_p$ (groupe des unit\'es).

\smallskip
Soit $i_K$ le plongement (surjectif) de ${\mathcal K}_1^\times$ dans 
$U_K=\bigoplus_{v \div p} U_v$ et soit
${\mathcal K}^\times_\infty$ le noyau de $i_K$ (groupe des 
infinit\'esimaux  \'etrangers \`a $S$ de Jaulent, \cite{Ja1}, \cite{Ja2}, \cite{Ja4}).
On d\'esigne par ${\mathcal P}_{K,\infty}$ le groupe des id\'eaux 
principaux $(x_\infty)$, $x_\infty \in {\mathcal K}^\times_\infty$.

\smallskip
On a alors (\cite{Gr1}, Th\'eor\`eme III.2.4)~:
$${\mathcal A}_K :=  {\rm Gal}(H_K^{\rm pr}/K) \simeq
 {\mathcal I}_K/ {\mathcal P}_{K,\infty},  {\mathcal B}_K := {\rm Gal}(H_K^{\rm pr}/H_K) 
\simeq {\mathcal P}_K/{\mathcal P}_{K,\infty} ,
{\mathcal T}_K := {\rm tor}_{\Z_p}({\mathcal A}_K) .$$

\begin{remark}\label{approx}
De fait, la th\'eorie du corps de classes d\'efinit 
${\mathcal A}_K = \ds \limproj_{n}\, {\rm Gal}(K^{(p^n)}/K) \simeq 
\limproj_{n} \, ( {\mathcal I}_K/{\mathcal P}_{K, (p^n)} )$, o\`u $K^{(p^n)}$ 
est le $p$-corps de rayon modulo $(p^n)$, o\`u ${\mathcal P}_{K, (p^n)} := \{(x),\ x\equiv 1 \pmod {p^n}\}$~; 
or $\ds\bigcap_n {\mathcal P}_{K, (p^n)} = {\mathcal P}_{K,\infty}$ (\cite{Gr1}, Proposition III.2.4.1). 

Le fait de pouvoir travailler avec des groupes d'id\'eaux \'etrangers \`a $S$ provient du th\'eor\`eme 
d'approximation id\'elique ou simplement du fait que toute classe d'id\'eaux contient un 
repr\'esentant \'etranger \`a $S$ (\cite{Gr1}, Th\'eor\`eme I.4.3.3 et Remarque I.5.1.2).
On a alors, avec des notations \'evidentes au niveau fini~: 
$${\rm Gal}(K^{(p^n)}/K) \simeq 
 {I}_{K,S_p}/{P}_{K,S_p, (p^n)} \simeq  {I}_{K,S}/{P}_{K,S, (p^n)} , $$
o\`u les isomorphismes avec le groupe de Galois
sont obtenus via le symbole d'Artin qui est d\'efini pour les id\'eaux 
\'etrangers \`a $S_p$, donc a fortiori \'etrangers \`a $S$.
\end{remark}

On d\'efinit les m\^emes objets relatifs \`a toute $p$-extension $L$ de $K$ 
en utilisant l'ensemble, not\'e encore $S$, des places de $L$ au-dessus de $S$.
Tous les objets ainsi obtenus par tensorisation avec $\Z_p$ sont des 
$\Z_p$-modules et se comportent comme les objets d'origine en raison de la platitude de~$\Z_p$. 

\smallskip
Nous commencerons par le cas $L/K$ cyclique qui est en un sens universel et qui
permet une plus grande effectivit\'e quant au plan num\'erique. 

\smallskip
On utilisera \`a plusieurs reprises les lemmes suivants~:

\begin{lemma}\label{leopoldt} Soit $k$ un corps de nombres v\'erifiant la conjecture de Leopoldt pour $p$.
Soit $\varepsilon \in {\mathcal E}_k := E_k \otimes \Z_p$ telle que, dans $U_k$, on ait $i_k(\varepsilon)^{p^e}=1$, 
$e \geq 0$(i.e., $i_k(\varepsilon) \in W_k$). Alors $\varepsilon \in \mu(k)\otimes \Z_p=\mu(k)$ et $\varepsilon^{p^e}=1$.
Il en r\'esulte ${\rm tor}_{\Z_p}(U_k/i_k(E_k)) = W_k/i_k(\mu(k))$.
\end{lemma}

\begin{proof} Caract\'erisation tr\`es classique utilisant 
l'injectivit\'e de $i_k$ sur ${\mathcal E}_k$ (\cite{Gr1}, Th\'eor\`eme III.3.6.2\,(vi) et \cite{Ja3}, Th\'eor\`eme 12).
\end{proof}

\begin{lemma}\label{zeta} 
Soit $L/K$ cyclique de degr\'e $p>2$ et soit $s$ un g\'en\'erateur de $G :={\rm Gal}(L/K)$.

\smallskip
(i) Si $\mu(K)=1$, alors $\mu(L)=1$ et ${\rm H}^1(G, \mu(L))=0$.

\smallskip
(ii) Si $\mu(K) =\mu(L)^p \ne 1$ (i.e., $L/K$ est cyclotomique globale de degr\'e~$p$), 
alors $\No_{L/K}(\mu(L)) = \mu(K)$ et $\mu(L)^{1-s}=\langle \zeta_1 \rangle$, o\`u $\zeta_1 \in \mu(K)$ est d'ordre $p$, 
d'o\`u ${\rm H}^1(G, \mu(L)) = 0$.

\smallskip
Si $\mu(K) = \mu(L) \ne 1$, $\No_{L/K}(\mu(L)) = \mu(K)^p$, $\mu(L)^{1-s}=1$, 
et ${\rm H}^1(G, \mu(L)) \simeq \Z/p\Z$. 

\smallskip
(iii) Si $\xi_L \in W_L$ v\'erifie $\xi_L^{1-s}= i_L(\zeta_L)$, o\`u 
$\zeta_L \in \mu(L)$, alors $\zeta_L$ est d'ordre 1 ou $p$. On a $W_L^G = W_K$ et $W_L^p \subseteq W_K$.
\end{lemma}

\begin{proof} (i) Si $\mu(K)=1$ et si $L$ contenait $\zeta_1$ 
d'ordre $p$,
on aurait $K \subset K(\zeta_1) \subseteq L$ o\`u $K(\zeta_1)/K$
serait de degr\'e \'egal \`a un diviseur de $p-1$ diff\'erent de 1 (absurde). 

\smallskip
(ii) Ce point est classique (e.g. \cite{Wa}).

\smallskip
(iii) Soit $\xi_L \in W_L$. On a $\xi_L^p\in W_{K}$ car si $\xi_{w} \in \mu(L_w)$ est une composante de 
$\xi_{L}$ telle que $\xi_{w}\notin \mu(K_{v})$ alors, 
comme pour le cas global, $\xi_{w}^p \in \mu(K_{v}$.
Comme $W_L = {\rm tor}_{\Z_p}(U_L)$ et que $U_L^G = U_K$, il  en r\'esulte que 
$W_L^G =  {\rm tor}_{\Z_p}(U_K) =W_K$.
\end{proof}

\section{Etude du ``transfert'' $j_{L/K} : {\mathcal C}_K \too {\mathcal C}_L$}

\subsection{G\'en\'eralit\'es}\label{gene} On suppose d\'esormais $p \ne 2$.
On s'int\'eresse au noyau du transfert $j_{L/K} : {\mathcal B\mathcal 
P}_K = {\rm tor}_{\Z_p} ({\mathcal C}_K) \too {\mathcal B\mathcal P}_L = {\rm tor}_{\Z_p}( {\mathcal C}_L)$
(notations et d\'efinitions du \S\,\ref{sch}),
dans une $p$-extension $L/K$, ult\'erieurement cyclique de degr\'e $p$.

\smallskip
Le transfert ${\mathcal C}_K \simeq {\mathcal A}_K/{\mathcal W_K} 
\too {\mathcal C}_L \simeq {\mathcal A}_L/ {\mathcal W_L}$ a m\^eme 
noyau (sous-groupe fini de ${\mathcal C}_K$, donc de torsion). On peut alors travailler
sur ce transfert plus simple not\'e encore $j_{L/K}$. 

\smallskip
Son noyau est aussi un noyau de capitulation de classes d'id\'eaux
puisque ${\mathcal A}_K \simeq {\mathcal I}_K/ {\mathcal P}_{K,\infty}$ et 
${\mathcal A}_L\simeq  {\mathcal I}_L/ {\mathcal P}_{L,\infty}$ sont des groupes de 
classes d'id\'eaux~; aussi en raison de la m\'ethode utilis\'ee, 
nous parlerons pour $j_{L/K}$ de {\it morphisme de capitulation}
et pour ${\rm Ker}(j_{L/K})$ de {\it noyau de capitulation}, comme dans \cite{Ja6}, \cite{Ng2}.

\smallskip
Faisons un rappel sur le morphisme g\'en\'eral de capitulation ${\mathcal A}_K\too {\mathcal A}_L $ en 
$p$-ramifi\-cation Ab\'elienne (\cite{Gr1}, Th\'eor\`eme IV.2.1)~:

\begin{lemma}\label{injectif} Dans une extension $L/K$ quelconque de 
corps de nombres, le morphisme de capitulation 
${\mathcal A}_K \to {\mathcal A}_L$ est  injectif sous la conjecture de Leopoldt pour $p$ et la cl\^oture galoisienne 
de $L$ sur $K$. Il en r\'esulte l'injectivit\'e de ${\mathcal T}_K = {\rm tor}_{\Z_p}({\mathcal A}_K) \too 
{\mathcal T}_L= {\rm tor}_{\Z_p}({\mathcal A}_L)$.
\end{lemma}

De fait, la premi\`ere preuve de ce r\'esultat a \'et\'e donn\'ee en 1982 dans \cite{Gr2}, Th\'eor\`eme I.1, d\'evelopp\'ee dans 
\cite{Gr3}, puis redonn\'ee dans d'autres cadres techniques, comme dans \cite{Gr1}, \cite{Ja1}, \cite{Ja2}, \cite{Ng1}, et constitue une 
propri\'et\'e classique, caract\'eristique de la conjecture de Leopoldt, qui pourrait simplifier consid\'erablement l'approche de 
\cite{Seo3}, \cite{Seo4}, \cite{Seo6}.

\subsection{Identification de ${\mathcal W_K}$ comme groupe de classes d'id\'eaux}\label{idw}
On a la suite exacte $1 \to {\mathcal K}^\times_\infty {\mathcal E}_K / {\mathcal E}_K 
\tooo {\mathcal K}_1^\times/ {\mathcal E}_K \ds\mathop{\tooo}^{i_K} U_K / \overline E_K \to 1$,
o\`u $\overline E_K = i_K( {\mathcal E}_K)$ est l'adh\'erence de $E_K$ dans $U_K$, et les isomorphismes~:
$$U_K / \overline E_K\simeq {\mathcal K}_1^\times/{\mathcal K}^\times_\infty {\mathcal E}_K 
\simeq  {\mathcal P}_K/{\mathcal P}_{K,\infty}  \simeq  {\mathcal B}_K = {\rm Gal}(H_K^{\rm pr}/H_K). $$

L'isomorphisme $U_K / \overline E_K \too {\mathcal P}_K/{\mathcal P}_{K,\infty}$ est donc ainsi d\'efini~:
\`a partir de $u \in U_K$ on prend $x \in {\mathcal K}_1^\times$ d'image $u$ par $i_K$ 
($x$ est d\'efini modulo ${\mathcal K}^\times_\infty$)
et on consid\`ere la classe de l'id\'eal principal $(x)$ modulo 
${\mathcal P}_{K,\infty}$, qui ne d\'epend pas du choix de $x$.
La surjectivit\'e est \'evidente, et si $u$ conduit \`a 
$(x)=(x_\infty) \in {\mathcal P}_{K,\infty}$, on a $x=x_\infty 
\varepsilon_K$, $\varepsilon_K \in {\mathcal E}_K$, d'o\`u 
$i_K(x)=u=i_K(\varepsilon_K) \in \overline E_K$. 

\smallskip
On suppose que la conjecture de Leopoldt est v\'erifi\'ee pour toute extension $L$ de $K$ consid\'er\'ee~:

\begin{lemma}\label{W} (i) L'image de ${\mathcal W_K} = W_K/i_K(\mu(K))$, 
dans ${\mathcal P}_K /{\mathcal P}_{K,\infty}$, est form\'ee des classes modulo ${\mathcal P}_{K,\infty}$ des id\'eaux de 
${\mathcal P}_K$ de la forme $(x)$, o\`u $i_K(x) = \xi_K \in W_K$, ce que l'on peut r\'esumer par l'isomorphisme
${\mathcal W_K} \simeq \{(x),\ x \in {\mathcal K}_1^\times,\  i_K(x) \in W_K \}\,.\, {\mathcal P}_{K,\infty} /{\mathcal P}_{K,\infty}$.

\smallskip
(ii) Le morphisme de capitulation $j_{L/K}$ est l'extension des ``classes d'id\'eaux'' ~:
$${\mathcal I}_K/ \{(x),\ x \in {\mathcal K}_1^\times,\  i_K(x) \in 
W_K \} \,.\, {\mathcal P}_{K,\infty} \too {\mathcal I}_L/\{(y),\ y \in {\mathcal L}_1^\times,\  i_{L}(y) \in 
W_L \} \,.\, {\mathcal P}_{L,\infty} . $$
\end{lemma}

\begin{proof} Le point (i) r\'esulte de la d\'efinition ci-dessus de $U_K / \overline E_K \too {\mathcal P}_K/{\mathcal P}_{K,\infty}$
appliqu\'ee au sous-groupe de torsion ${\rm tor}_{\Z_p}(U_K/ \ov E_K) = {\mathcal W}_K$ (Lemme \ref{leopoldt}).

\smallskip
 Il en r\'esulte que
${\mathcal C}_K \simeq {\mathcal A}_K/ {\mathcal W_K} \simeq {\mathcal I}_K/
\{(x),\ x \in {\mathcal K}_1^\times,\  i_K(x) \in W_K \} \,.\, {\mathcal P}_{K,\infty}$,
d'o\`u facilement le point (ii) du lemme.
On rappelle que ${\mathcal I}_K$ et ${\mathcal I}_L$ sont form\'es 
d'id\'eaux \'etrangers \`a $S$ ($S$ contenant $S_p$ et les places ramifi\'ees dans $L/K$).
\end{proof}

\subsection{Caract\'erisation du noyau de capitulation (cas cyclique de degr\'e $p$)}\label{ker}
Soit ${\mathfrak a} \in {\mathcal I}_K$ tel que par extension, $({\mathfrak a})_L= (y)(y_\infty)$ o\`u 
$y \in {\mathcal L}_1^\times$, $y_\infty \in {\mathcal L}^\times_\infty$, 
avec $i_{L}(y)  = \xi_L \in W_L$ et $i_{L}(y_\infty)  =1$ (Lemme \ref{W}\,(ii)). 

\medskip
A partir de $({\mathfrak a})_L= (y\,.\,y_\infty)$, la norme $ \No_{L/K}$ conduit en id\'eaux \`a  la relation
${\mathfrak a}^p = (\No_{L/K}(y \,.\, y_\infty) )$~; posons
$\No_{L/K}(y) = x \in  {\mathcal K}_1^\times$ et $\No_{L/K}(y_\infty)=x_\infty \in  {\mathcal K}_\infty^\times$, 
ce qui donne dans $L$ la relation $y^p \,.\, y_\infty^p = x\,.\,x_\infty\,.\,\varepsilon_L$, $\varepsilon_L \in {\mathcal E}_L$.
Mais $i_L(y^p) = i_L( x) \,.\,  i_{L}(\varepsilon_L)$, ou encore
$i_L(y^p) =\No_{L/K}( i_L( y) )\,.\,  i_{L}(\varepsilon_L)$ qui s'\'ecrit par hypoth\`ese
$\xi_L^p = \No_{L/K}( \xi_L )\,.\,  i_{L}(\varepsilon_L)$~; donc (Lemme \ref{zeta}, preuve de (iii)), on a $\xi_L^p\in W_K$ et
$i_{L}(\varepsilon_L)\in W_K$, et d'apr\`es le Lemme \ref{zeta} on a $\varepsilon_L = \zeta_K \in \mu(K)$.

\smallskip
On a donc en posant $\alpha :=  x \,.\, x_\infty \,.\,\zeta_K \in  {\mathcal K}_1^\times$, la relation
$\alpha = (y \,.\, y_\infty)^p$ que l'on interpr\`ete de la fa\c con suivante~:

\smallskip
(i) Cas $y \,.\, y_\infty  = x' \in {\mathcal K}_1^\times$. On a $i_L(y) = i_L(x')=\xi_L \in W_L \, \cap\, U_L^G$, donc 
$i_L(y) =\xi_K \in W_K$ et $y$ est de la forme $y = x'' \,.\, y'_\infty$ avec $i_K(x'')=\xi_K$,
d'o\`u $x'= y \,.\, y_\infty = x'' \,.\, y''_\infty$~; donc $y''_\infty = x''_\infty \in {\mathcal K}_\infty^\times$
et il vient $({\mathfrak a})_L= (x'' \,.\, x''_\infty)_L$ qui se redescend en 
${\mathfrak a}= (x'' \,.\, x''_\infty)$~; or $i_L(x'') = \xi_K \in W_K$, d'o\`u la trivialit\'e de la classe
de ${\mathfrak a}$ dans ${\mathcal C}_K$ (Lemme \ref{W}\,(i)).

\medskip
(ii) Cas $ y \,.\, y_\infty \notin {\mathcal K}_1^\times$. On est n\'ecessairement dans un cas kummerien o\`u
$L=K(\sqrt[p]{\alpha})$ et $\mu(K) \ne 1$ puisque $(y \,.\, y_\infty)^p = \alpha \in {\mathcal K}_1^\times$
(de fait on prend pour $\alpha$ un repr\'esentant de $x \,.\, x_\infty \,.\,\zeta_K$,
modulo ${\mathcal K}_1^{\times p}$, dans $K_1^\times$, afin de donner un sens \`a l'extension kummerienne $L$). 
Comme $\alpha$ est par d\'efinition \'etranger \`a $S$, et que la ramification mod\'er\'ee se lit sur l'id\'eal $(\alpha)$, 
il en r\'esulte que $L/K$ est $p$-ramifi\'ee. 

\medskip
On d\'eduit de ce qui pr\'ec\`ede qu'en dehors du cas (ii), il ne peut y avoir capitulation non triviale et
on a donc obtenu le r\'esultat partiel suivant remarqu\'e dans \cite{Seo6}

\begin{theorem} \label{trivial} Soit $L/K$ une $p$-extension.
Sous la conjecture de Leopoldt pour $L$ et $p>2$, et si $\mu(K)=1$, le morphisme de capitulation 
$j_{L/K} : {\mathcal B\mathcal P}_K \too {\mathcal B\mathcal P}_L$ 
est injectif.
\end{theorem}

\begin{proof}
Comme $L/K$ est une $p$-extension, on est ramen\'e au cas cyclique ci-dessus en 
remarquant que les hypoth\`eses se transmettent \`a tout corps $k$ 
compris entre $K$ et $L$ ($\mu(k)=1$ et conjecture de Leopoldt).
\end{proof}

Une condition n\'ecessaire pour ${\rm Ker}(j_{L/K}) \ne 1$ est donc 
$\mu(K) \ne 1$ et  $L/K$  $p$-ramifi\'ee.  

\smallskip
Poursuivons l'\'etude du noyau de capitulation avec la seule hypoth\`ese
$L/K$ $p$-ramifi\'ee et  par cons\'equent $S=S_p$ 
(l'hypoth\`ese kummerienne n'est pas n\'ecessaire pour l'instant car on la retrouvera dans le calcul du noyau).
On revient \`a  ${\mathfrak a} \in {\mathcal I}_K$ tel que par extension, 
$({\mathfrak a})_L= (y \,.\, y_\infty)$ o\`u $y \in {\mathcal L}_1^\times$, 
$y_\infty \in {\mathcal L}^\times_\infty$, avec $i_{L}(y)  = \xi_L \in W_L$ et $i_{L}(y_\infty)  =1$. 

\smallskip
Soit $s$ un g\'en\'erateur de $G := {\rm Gal}(L/K)$. Il vient alors $({\mathfrak a})_L^{1-s}= (1) = (y^{1-s}) 
(y_\infty^{1-s})$, d'o\`u $y^{1-s} y_\infty^{1-s} = \zeta_K$ o\`u $\zeta_K \in \mu(K)$ est d'ordre $1$ ou $p$ puisque
$(y\, .\, y_\infty)^p = \alpha \in  {\mathcal K}^\times_1$.

\smallskip
R\'eciproquement, si $\xi_L \in W_L$ est telle que
$\xi_L^{1-s}=i_L(\zeta_K)$, en prenant $y \in {\mathcal L}_1^\times$ tel que
$i_L(y)= \xi_L$, il vient $i_{L}(\zeta_K) = i_L(y)^{1-s}$, d'o\`u 
$\zeta_K = y^{1-s} y'_\infty$, $y'_\infty \in {\mathcal L}^\times_\infty$~;
on a $\No_{L/K}(y'_\infty )= \No_{L/K}(\zeta_K) =1$ d'apr\`es ce qui pr\'ec\`ede~; le 
Th\'eor\`eme 90 de Hilbert--Speiser--N\oe ther dans ${\mathcal L}^\times_\infty$
(\cite{Gr1}, preuve du Lemme IV.3.1), 
implique $y'_\infty = y_\infty^{1-s}$, $y_\infty \in {\mathcal L}^\times_\infty$, 
et on obtient une relation de la forme $\zeta_K = y^{1-s} y_\infty^{1-s}$ ce qui conduit \`a 
$(y \,y_\infty) \in {\mathcal I}_L^G=({\mathcal I}_K)_L$ car il n'y a pas de ramification en dehors de $p$.
 L'id\'eal $(y \, y_\infty)$ de $L$ est l'\'etendu d'un id\'eal ${\mathfrak a}$ de $K$ \'etranger \`a $p$
dont la classe est dans le noyau de $j_{L/K}$. On a donc \`a ce stade~:

\medskip\noindent
${\rm Ker}(j_{L/K}) = \{{\mathfrak a} \in {\mathcal I}_K, \, 
({\mathfrak a})_L= (y)(y_\infty),   i_L(y)=\xi_L \    \&\   \xi_L^{1-s} \in i_L(\mu(K)) \} \,.\, {\mathcal P}_{K,\infty} $

\hfill $\,\big / \,\{(x),\ x \in {\mathcal K}_1^\times,\  i_K(x) \in W_K 
\}\,.\, {\mathcal P}_{K,\infty}$,

\smallskip\noindent
qu'il reste \`a identifier.

\medskip
Soit ${\mathfrak a} \in {\mathcal I}_K$ que $({\mathfrak a})_L= (y)(y_\infty)$, $i_{L}(y)  = \xi_L \in W_L$,  
$\xi_L^{1-s} = i_L(\zeta_K)$~; si  ${\mathfrak a} = (x\,x_\infty)$ avec $i_K(x)=\xi_K \in W_K$ et $i_K(x_\infty)=1$,
il vient $i_L(y)=i_L(x) i_L(\varepsilon_L)$, $\varepsilon_L \in {\mathcal  E}_L$, soit $\xi_L=\xi_K \, i_L(\varepsilon_L)$~;
donc  $i_L(\varepsilon_L)$ est de la forme $i_L(\zeta_L)$, $\zeta_L \in \mu(L)$,
d'o\`u $\xi_L \in W_K\, i_L(\mu(L))$.

\smallskip
R\'eciproquement, si $\xi_L=\xi_K i_L(\zeta_L)\in W_K\, i_L(\mu(L))$
(qui implique $\xi_L^{1-s} = i_L(\zeta_K)$),
on a $i_L(y) = i_L(x)\, i_L(\zeta_L)$ pour $x \in {\mathcal K}_1^\times$, 
d'o\`u, en id\'eaux, $(y) = (x)_L (y'_\infty)$ et $({\mathfrak a})_L = (x)_L 
(y''_\infty)$, o\`u $(y''_\infty) \in {\mathcal P}_{L,\infty}^G=
{\mathcal P}_{K,\infty}$, d'o\`u ${\mathfrak a} = (x)\,(x_\infty)$
avec $i_K(x) \in W_K$.

\smallskip
 La classe de ${\mathfrak a}$ dans le noyau de $j_{L/K}$ est nulle si 
et seulement si dans l'\'ecriture 
$({\mathfrak a})_L= (y)(y_\infty)$, on a $i_L(y) \in W_K\, 
i_L(\mu(L))$.
On a donc obtenu une description du noyau de capitulation dans le cas $L/K$ cyclique $p$-rami\-fi\'ee de degr\'e $p$
(analogue du Th\'eor\`eme 2.1 de \cite{Ng2} et Proposition 8 de \cite{Ja6})~:

\begin{lemma}\label{cns} On a ${\rm Ker}(j_{L/K}) \simeq \{\xi_L \in 
W_L, \  \xi_L^{1-s} \in i_L(\mu(K))\} / W_K\, i_L(\mu(L))$.
Si de plus $\mu(K)=1$, on a $i_L(\mu(L))=1$ et $\{\xi_L \in 
W_L, \  \xi_L^{1-s} =1\}=W_K$.
\end{lemma}

\begin{proof} Il suffit de consid\'erer l'application qui \`a 
${\mathfrak a}$, dont la classe est dans noyau de capitulation, associe
$\xi_L$ (qui est d\'efinie modulo $\mu(L)$).
\end{proof}

Si $\mu(K)=1$, le lemme redonne le Th\'eor\`eme \ref{trivial} sur la nullit\'e du noyau de capitulation.

\subsection{Etude du noyau de capitulation dans le cas kummerien}

On suppose d\'esor\-mais que $\zeta_1$ d'ordre $p$ est dans $K$. On a, \`a partir du Lemme \ref{cns},
la suite exacte~:

\smallskip
\centerline{$1 \to {\rm Ker}(j_{L/K}) \too W_L/W_K\, i_L(\mu(L))  \ds 
\mathop{\tooo}^{1-s}  W_L^{1-s}/ W_L^{1-s} \cap  i_L(\mu(K))\to 1$,}

\medskip\noindent
qui conduit, puisque  $W_L^{1-s} \simeq W_L/W_K$ et $W_K \,i_L(\mu(L)) / W_K \simeq \mu(L) /\mu(K)$ \`a~:

\begin{lemma}\label{kummer}
Dans le cas kummerien $p$-ramifi\'e cyclique de degr\'e $p$, et sous la conjecture 
de Leopoldt pour $L$ et $p$, on a $\vert  {\rm Ker}(j_{L/K}) \vert = 
\vert  W_L^{1-s} \cap  i_L(\mu(K)) \vert \times \Frac{\vert \mu(K) \vert}{\vert \mu(L) \vert}$, o\`u 
$\vert  W_L^{1-s} \cap  i_L(\mu(K)) \vert$  est d'ordre $1$ ou $p$.

\smallskip\noindent
Si $\mu(K) = \mu(L) \ne 1$, on a 
$\vert  {\rm Ker}(j_{L/K}) \vert = \vert  W_L^{1-s} \cap  i_L(\mu(K)) 
\vert$.

\smallskip\noindent
Si $\mu(K) = \mu(L)^p \ne 1$ (i.e. $L/K$ est 
cyclotomique globale), $j_{L/K}$ est injectif. 
\end{lemma}

\begin{proof}
En effet, d'apr\`es le Lemme \ref{zeta} on a $W_L^{1-s} \cap  i_L(\mu(K)) \subseteq \langle i_L(\zeta_1) \rangle$ 
qui est d'ordre 1 ou~$p$.
Ainsi il y a injectivit\'e de $j_{L/K}$ si et seulement si $L/K$ est cyclotomique globale ou bien si
$\mu(K) = \mu(L) \ne 1$ et $i_L(\zeta_1) \notin W_L^{1-s}$.
\end{proof}

Ceci peut se caract\'eriser de la fa\c con concr\`ete suivante~:

\begin{theorem}\label{cns2} Soit $L/K$ une extension cyclique, $p$-ramifi\'ee, de degr\'e $p>2$, 
v\'erifiant la conjecture de Leopoldt pour $p$.

\smallskip\noindent
Lorsque $\mu(K) = \mu(L) \ne 1$ (i.e., $L/K$ kum\-merienne non globalement cyclotomique), 
le morphisme de capitulation $j_{L/K} : {\mathcal B\mathcal P}_K \too {\mathcal B\mathcal P}_L$ 
est injectif si et seulement si il existe $v_0 \div p$ dans $K$, non d\'ecompos\'ee dans $L$,  telle que 
$\mu(L_{w_0}) = \mu(K_{v_0})$ pour l'unique $w_0 \div v_0$
(i.e., $L/K$ est, en $v_0$, localement de degr\'e $p$ et non localement cyclotomique).\,\footnote{Il r\'esulte de ceci que ``$L/K$
$p$-ramifi\'ee  \&  localement cyclotomique en $p$'' est \'equivalente \`a ``$L/K$ partout localement cyclotomique'' puisque 
$L_w/K_v$ est, pour $v\notdiv p$, le rel\`evement $v$-adique de l'extension r\'esiduelle $F_w/F_v$, 
qui est cyclotomique (que $v$ soit d\'ecompos\'ee ou inerte). Dans \cite{Ja6} et \cite{Ng2} c'est le vocabulaire adopt\'e
syst\'ematiquement. }
\end{theorem}

\begin{proof} Ecrivons $W_L = \bigoplus_{v \div p} W_{L, v}$ en un sens \'evident. 
On a alors trois cas selon que $v$ est d\'ecompos\'ee dans $L/K$ ou non,
puis que $L_w/K_v$ est localement cyclotomique ou non~; on calcule alors 
${\rm H}^1(G, W_{L, v})$ (ou ${\rm H}^2(G, W_{L, v})$ plus facile) pour obtenir $\vert  W_{L, v}^{1-s}\vert$~:

\smallskip
(i) Cas d\'ecompos\'e, i.e.,  $W_{L, v} = \bigoplus_{w \div v} 
\mu(K_v)$~:  on a ${\rm H}^1(G, W_{L, v}) = 0$ (Lemme 
de Shapiro), d'o\`u $W_{L, v}^{1-s} = 
{}_N W_{L,v}$ qui contient l'image de $\langle \zeta_1 \rangle$.

\smallskip
(ii) Cas non d\'ecompos\'e et $\mu(L_w)^p =\mu(K_v)$~: on a  ${\rm H}^1(G, W_{L, v}) =0$
et $W_{L, v}^{1-s} = \langle i_v(\zeta_1) \rangle$.

\smallskip
(iii)  Cas non d\'ecompos\'e et $\mu(L_w) =\mu(K_v)$~:  on a ${\rm H}^1(G, W_{L, v}) \simeq \Z/p\Z$
et $\vert  W_{L, v}^{1-s}\vert = 1$.

\medskip
Le th\'eor\`eme en r\'esulte puisque la condition contraire $i_L(\zeta_1) \in W_L^{1-s}$est \'equivalente 
\`a la non nullit\'e de tous les $W_{L, v}^{1-s}$. Seul  (iii) conduit \`a $\vert  W_L^{1-s} \cap i_L(\mu(K)) \vert = 1$.
\end{proof}

\section{Interpr\'etation via un radical kummerien}
On suppose que $\mu(K)\ne 1$ et on pose $L=K(\sqrt[p] \alpha)$, 
$\alpha \in K^\times$~; on suppose que $L/K$ est $p$-ramifi\'ee et v\'erifie la conjecture de Leopoldt pour $p$. 
On a $(\alpha)={\mathfrak a}_0^p\, {\mathfrak a}_{p}$, o\`u ${\mathfrak a}_0$ est un id\'eal de $K$ \'etranger \`a $p$
et o\`u ${\mathfrak a}_{p}$ est un produit d'id\'eaux premiers au-dessus de $p$ \`a une puissance $< p$.
On pose $(\sqrt[p] \alpha)^{1-s} = \zeta_1$ d'ordre~$p$.

\smallskip
On suppose, dans cette \'etude, que $L/K$ n'est pas (globalement) cyclotomique
(on est dans le cas $\mu(L)=\mu(K) \ne 1$).
D'apr\`es le Th\'eor\`eme \ref{cns2}, $j_{L/K}$ est {\it non injectif} 
si et seulement si pour toute place $v \div p$, non d\'ecompos\'ee 
dans $L/K$, $\mu(K_v) = \mu(L_w)^p$ pour l'unique place $w$ de $L$ au-dessus de $v$.
Dans ce cas, il existe donc ${\mathfrak a} \in {\mathcal I}_K$,
\'etranger \`a $p$, dont la classe est dans le noyau de $j_{L/K}$ et 
est {\it non triviale}~; on a $({\mathfrak a})_L=(y \, y_\infty)$, avec $i_L(y)=\xi_L$, 
$i_L(y_\infty)=1$, et (par analogie avec les calculs du \S\,\ref{ker})~:
$$i_L(y)^{1-s} =  \xi_L^{1-s}= i_L(\sqrt[p] \alpha)^{1-s} = i_L(\zeta_1). $$

Soit  $v \div p$ fix\'ee non d\'ecompos\'ee (le cas d\'ecompos\'e \'equivaut au fait que 
$\alpha$ est puissance $p$-i\`eme locale en $v$)~; on a, avec des 
notations \'evidentes, $i_w(y) = \xi_w$ et $i_w(y)^{1-s} =  \xi_w^{1-s}=i_w(\zeta_1)$ et par cons\'equent 
on peut prendre $\xi_w$ comme radical local, autrement dit, $L_w = K_v(\sqrt[p] {\xi_v})$ 
o\`u $\xi_w^p =: \xi_v \in \mu(K_v)$ et il en r\'esulte, par ``unicit\'e'' d'un radical, que ${\mathfrak a}_{p} = 1$.
D'o\`u $(\alpha)={\mathfrak a}_0^p$, ce qui s'\'ecrit, dans $L$, $(\sqrt[p] \alpha) = ({\mathfrak a}_0)_L$.

\smallskip
On a $i_L(y) = i_L(\sqrt[p] \alpha) \, a,\ \, a\in {\mathcal K}_1^\times$, 
$a$ \'etranger \`a $p$. Quitte \`a modifier $\alpha$ modulo ${\mathcal K}_1^{\times p}$ et 
${\mathfrak a}_0$ en cons\'equence, on peut supposer que 
$i_L(y) = i_L(\sqrt[p] \alpha) = \xi_L$ avec $\xi_L^{1-s}=i_L(\zeta_1)$, auquel cas
$({\mathfrak a})_L=({\mathfrak a}_ 0 )_L \, (y_\infty)$. 
Comme l'id\'eal $(y_\infty)$ est invariant par $G$, il est de la forme $(x_\infty)$, $x_\infty \in
{\mathcal K}_\infty^\times$ (\cite{Gr1}, Lemme IV.3.1). On a alors la relation
${\mathfrak a}= {\mathfrak a}_0 \pmod{ {\mathcal P}_{K, \infty}}$.

\medskip
On peut donc \'enoncer le r\'esultat suivant  (analogue aux formulations donn\'ees par \cite{Ja6} et \cite{Ng2}),
o\`u l'on rappelle que pour un corps $k$, $\mu(k)$ est le $p$-groupe des racines de l'unit\'e 
de $k$, que $\widetilde k$ est le compos\'e des $\Z_p$-extensions de $k$, que $U_k$ est le 
$\Z_p$-module d'unit\'es locales principales en $p$ de $k$ et que $W_k = {\rm tor}_{\Z_p}(U_K)$~; 
on d\'esigne par $BP_k$ la pro-$p$-extension Ab\'elienne maximale de $k$ compo\-s\'ee des 
$p$-extensions cycliques plongeables dans une $p$-extension cyclique de $k$ de 
degr\'e arbitrairement grand~; on pose ${\mathcal B\mathcal P}_k:={\rm Gal}(BP_k/ \widetilde k)$~:

\begin{theorem} \label{thmf}
Soit $L/K$ une extension cyclique de corps de nombres, de degr\'e premier $p>2$.
On suppose que $L$ v\'erifie la conjecture de Leopoldt pour $p$.

\smallskip\noindent
Alors le morphisme de capitulation $j_{L/K} :  {\mathcal B\mathcal P}_K \too {\mathcal B\mathcal P}_L$
est injectif sauf si les deux conditions suivantes sont satisfaites~:

\smallskip
(i) $\mu(L) =\mu(K) \ne 1$, $L = K(\sqrt[p] \alpha)$, $\alpha \in K^\times$ et 
$(\alpha)= {\mathfrak a}_0^p$, o\`u ${\mathfrak a}_0$ est un id\'eal de $K$ \'etranger \`a $p$ 
(i.e., $L/K$ est kummerienne, non globalement cyclotomique, et $p$-ramifi\'ee (i.e., non ramifi\'ee en dehors de $p$))~;

\smallskip
(ii) l'image diagonale de $\alpha$ dans $U_K$ est \'egale \`a $\xi_K \,u_K^p$, $\xi_K\in W_K$, $u_K\in U_K$,  
avec $\xi_K \in W_L^p$ (i.e., $L/K$ est localement cyclotomique en toute place $v \div p$).

\medskip\noindent
{\rm \bf English version.}
Let $L/K$ be a cyclic extension of number fields, of prime degree $p>2$.
We suppose that $L$ satisfies the Leopoldt conjecture for $p$.

\smallskip\noindent
Then the morphism of capitulation 
$j_{L/K} : {\mathcal B\mathcal P}_K \too {\mathcal B\mathcal P}_L$
is injective except if the two following conditions are satisfied~:

\smallskip
(i) $\mu(L) = \mu(K) \ne 1$, $L = K(\sqrt[p] \alpha)$, $\alpha \in K^\times$ and 
$(\alpha)={\mathfrak a}_0^p$, where ${\mathfrak a}_0$ is a prime to $p$ ideal of $K$ (i.e., $L/K$ is kummerian,
non globally cyclotomic, and $p$-ramified (i.e., unramified outside $p$))~;

\smallskip
(ii) the diagonal image of $\alpha$ into $U_K$ is equal to $\xi_K \,u_K^p$, $\xi_K\in W_K$, $u_K\in U_K$,  
with  $\xi_K \in W_L^p$ (i.e., $L/K$ is localy cyclotomic at each place $v \div p$). 
\end{theorem}

On pourrait alors obtenir le r\'esultat plus g\'en\'eral de Jaulent (\cite{Ja6}, Corollaire 11)~:

\begin{theorem} 
Dans une $p$-extension $L/K$ de corps de nombres, contenant les racines $p$-i\`emes de l'unit\'e,
et satis\-faisant \`a la conjecture de Leopoldt pour le nombre premier $p>2$, le sous-groupe ${\rm Ker}(j_{L/K})$ 
des \'el\'ements de ${\mathcal B\mathcal P}_K$ qui capitulent dans ${\mathcal B\mathcal P}_L$ est d'ordre 
${\rm min} \,\{\vert \mu(K) \vert, [L\cap K^{\rm lc} : K(\mu(L))]\}$, o\`u $K^{\rm lc}$ 
est la pro-$p$-extension localement cyclotomique maximale de $K$.
\end{theorem}

\section{Exemples num\'eriques}
 On prend pour $K$ un corps biquadratique contenant le corps $\Q(j)$ 
des racines cubiques de l'unit\'e et on consid\`ere $p=3$  (voir des exemples analogues dans \cite{Ja6}, \S\,3)~:

\smallskip
(i) Soient $K= \Q(\sqrt{-23}, j)$ et l'unit\'e $\alpha =j\, \frac{ 25+3\sqrt {69}}{2}$.
On consid\`ere $L = K(\sqrt[3]\alpha)$~; on a bien $\mu(L) = \mu(K) \ne 1$.
Comme $3$ est d\'ecompos\'e dans $K$ en deux id\'eaux premiers 
${\mathfrak p}, \,{\mathfrak p}'$, les deux compl\'et\'es de $K$ sont \'egaux \`a $\Q_3(j)$.

\smallskip
On a $\alpha . j^{-1}= \frac{ 25+3\sqrt {69}}{2} \equiv -1 \pmod {({\mathfrak p}{\mathfrak p}' )^3}$, 
ce qui montre que les deux compl\'et\'es de $L$ sont \'egaux \`a $\Q_3 (\omega )$, o\`u
$\omega$ est une racine de l'unit\'e d'ordre 9~; on est dans le cas 
$\mu(K_v)=\mu(L_w)^p$ pour toute place $v \div p$, et les conditions du Th\'eor\`eme 
\ref{thmf} sont donc satisfaites puisque
$i_K(\alpha)=i_K( j) \,.\,  i_K \big(-1+\frac{3 \sqrt {69}}{2} \big)  = \xi_K \,.\, u_K^3$, avec
$\xi_K = (j, j) = (\omega, \omega)^3$, le nombre $-1 +\frac{3\sqrt {69}}{2}$ \'etant un cube dans $U_K$~; donc
 $j_{L/K}$ est non injectif.

\smallskip
(ii) On se place \`a nouveau dans $K= \Q(\sqrt{-23}, j)$ et on consid\`ere l'unit\'e 
$\alpha = \frac{25+3\sqrt {69}}{2}$ et  $L = K(\sqrt[3]\alpha)$. Comme $\alpha$ est localement un cube, 
$L/K$ est d\'ecompos\'ee en $3$ et on est dans le cas $\mu(K_v)=\mu(L_w)$. Donc $j_{L/K}$ est injectif.

\section{Points fixes du module de Bertrandias--Payan}

Dans cette section, $L/K$ est une $p$-extension quelconque de corps de nombres, de groupe de 
Galois $G$, que l'on supposera cyclique de degr\'e $p$ \`a partir du \S\,\ref{cycl}. 
On omet le plus souvent les plongements diagonaux $i_K$ et $i_L$.
Ici il n'y a aucune hypoth\`ese sur la ramification dans $L/K$.

On a la suite exacte qui d\'efinit ${\mathcal B\mathcal P}_L$ \`a partir de
${\mathcal T}_L = {\rm Gal}(H_L^{\rm pr} / \widetilde L)$~:
$$1 \to W_L/\mu(L) \too {\mathcal T}_L \too {\mathcal B\mathcal P}_L 
\to 1, $$
 o\`u $W_L = \bigoplus_{v \div p} W_{L,v}$, avec $W_{L,v} =\bigoplus_{w \div v} \mu(L_w)$ 
pour toute place $v \div p$ de $K$.

\smallskip
On a de m\^eme la suite exacte 
$1 \to W_K/\mu(K) \too {\mathcal T}_K \too {\mathcal B\mathcal P}_K \to 1$.

\subsection{Suites exactes fondamentales}\label{cohomo}
On suppose que $L$ v\'erifie la conjecture de Leo\-poldt pour $p$, ce 
qui implique que ${\mathcal T}_L^G$ contient un sous-groupe isomorphe \`a ${\mathcal T}_K$ 
(injectivit\'e du morphisme de capitulation ${\mathcal T}_K \too {\mathcal T}_L$, Lemme \ref{injectif}).
On a la suite exacte~:
$$\hspace{1.9cm} 1 \to (W_L/\mu(L))^G \too {\mathcal T}_L^G \too 
{\mathcal B\mathcal P}_L^G \ds \mathop{\too}^{\theta}
{\rm H}^1(G, W_L/\mu(L)) , \hspace{1.9cm}  (1)\ \, $$

\medskip\noindent
et \`a partir de $1 \to \mu(L)  \too W_L \too W_L/\mu(L) \to 1$, on a 
la suite exacte~:
$$\hspace{1.4cm}  1 \to W_K/\mu(K)  \too (W_L/\mu(L))^G 
\mathop{\too}^{\psi}  {\rm H}^1(G, \mu(L))
\mathop{\too}^{\nu}  {\rm H}^1(G, W_L),\hspace{1.1cm}  (2)$$
avec ${\rm H}^1(G, W_L) =\bigoplus_{v \div p} {\rm H}^1(G, W_{L,v})$.

\medskip
Dans le cas $G$ cyclique d'ordre $p$ (engendr\'e par $s$) on aura les suites exactes~:
$$\hspace{1.1cm} 1 \to (W_L/\mu(L))^G \too {\mathcal T}_L^G \too 
{\mathcal B\mathcal P}_L^G \ds \mathop{\too}^{\theta}
 {}_N^{}(W_L/\mu(L)) \big / (W_L/\mu(L))^{1-s} , \hspace{1.0cm}  
(1')\ \, $$
$$\hspace{0.4cm}  1 \to W_K/\mu(K)  \too (W_L/\mu(L))^G 
\mathop{\too}^{\psi}  {}_N^{}\mu(L) \big /\mu(L)^{1-s} 
\mathop{\too}^{\nu} \hbox{$\bigoplus_{v \div p}\ $}  {}_N^{}W_{L,v}/ 
W_{L,v}^{1-s} .\hspace{0.4cm}  (2')\ \,$$

\begin{remark} Dans le cas $G$ cyclique d'ordre $p$, les noyaux de la ``norme'' $N = N_G$ sont ceux de la 
norme alg\'ebrique (ici $N_G = 1+s+\cdots +s^{p-1}$). 
Or on peut utiliser, par commodit\'e des calculs, la norme arithm\'etique $\No_{L/K}$ 
\`a condition que pour les objets $X_L$ et $X_K$, l'application naturelle $i_{L/K} : X_K \too X_L$ soit injective puisque
$i_{L/K} \circ \No_{L/K} = N_G$. 

\noindent
C'est le cas ici pour les objets du type $\mu(k)$, $W_k$, ${\mathcal W}_k = W_k/\mu(k)$, 
${\mathcal T}_k$, mais non n\'ecessaire\-ment pour ${\mathcal B\mathcal P}_k$.
\end{remark}

On souhaite voir, comme dans \cite{Seo3}, \cite{Seo4}, \cite{Seo6},
dans quels cas on a $\vert {\mathcal B\mathcal P}_L \vert  \geq \vert {\mathcal B\mathcal P}_K \vert$ 
par le fait que par exemple $\vert {\mathcal B\mathcal P}_L^G\vert \geq \vert {\mathcal B\mathcal P}_K \vert$. 
D'apr\`es (1) et (2) il vient~:
$$\hspace{1.1cm} \vert {\mathcal B\mathcal P}_L^G \vert = \frac{\vert 
{\mathcal T}_L^G\vert \,.\, \vert {\rm Im}(\theta)\vert}{\vert (W_L/\mu(L))^G\vert }=
\frac{\vert {\mathcal T}_L^G\vert \,.\,\vert {\rm Im}(\theta)\vert}{\vert W_K/\mu(K)\vert \,.\, 
\vert {\rm Im}(\psi)\vert }=\frac{\vert {\mathcal T}_L^G\vert }{\vert  {\mathcal T}_K\vert} \,.\,
\vert {\mathcal B\mathcal P}_K \vert \,.\, \frac{\vert {\rm Im}(\theta)\vert }{\vert {\rm Im}(\psi)\vert} . 
\hspace{0.4cm} (3)\  \ $$

On obtient donc, en utilisant la formule des points fixes (\cite{Gr1}, Th\'eor\`eme IV.3.3)~:

\begin{theorem}\label{ptsfixes} Soit $L/K$ une $p$-extension v\'erifiant la conjecture de Leopoldt pour $p$.
On a la formule suivante (o\`u le premier facteur est entier)~:
$$\vert {\mathcal B\mathcal P}_L^G \vert =\frac{ \prod_{{\mathfrak q}\,\notdiv\, p} e^{}_{\mathfrak q}}
{\Big( \sm_{{\mathfrak q}\,\notdiv\, p}\, \hbox{$\frac{1}{e^{}_{\mathfrak q}}$}  \Z_p {\rm Log}_{K}({\mathfrak q}) + 
\Z_p {\rm Log}_{K}(I_{K}) : \Z_p {\rm Log}_{K}(I_{K}) \Big)}
\,\cdot\, \vert {\mathcal B\mathcal P}_K \vert \, \cdot \, \frac{\vert {\rm Im}(\theta)\vert }{\vert {\rm Im}(\psi)\vert} , $$
o\`u $e^{}_{{\mathfrak q}}$ est l'indice de ramification de ${\mathfrak q}$ dans $L/K$, et
${\rm Log}_{K}  : I_{K}\to \bigoplus_{v\,|\,p} K_v /\Q_p {\rm log}_{K}(E_K)$, la fonction logarithme 
d\'efinie dans \cite{Gr1}, III.2.2 (voir aussi le \S\,\ref{rho}).
\end{theorem}

\begin{corollary}
On peut trouver une infinit\'e d'extensions $L/K$ telles que ${\mathcal B\mathcal P}_L^G$ 
soit non trivial et m\^eme telles que $\vert {\mathcal B\mathcal P}_L^G \vert =
\prd_{{\mathfrak q}\,\notdiv\, p} e^{}_{\mathfrak q} \,\cdot\, \vert {\mathcal B\mathcal P}_K \vert \, \cdot \, 
\Frac{\vert {\rm Im}(\theta)\vert }{\vert {\rm Im}(\psi)\vert}$.
\end{corollary}

\begin{corollary} Lorsque l'extension $L/K$ est $p$-primitivement ramifi\'ee  (\cite{Gr1}, IV.3, \S\,(b), D\'efinition 3.4), 
ce qui est (sous la conjecture de Leopoldt pour $L$ et $p$) \'equivalent \`a
$\big( \sum_{{\mathfrak q}\,\notdiv\, p}\,\hbox{$\frac{1}{e^{}_{\mathfrak q}}$}  
\Z_p {\rm Log}_{K}({\mathfrak q}) + \Z_p {\rm Log}_{K}(I_{K}) : \Z_p {\rm Log}_{K}(I_{K}) \big) =  
\prod_{{\mathfrak q}\,\notdiv\, p} e^{}_{\mathfrak q}$,
on a $\vert {\mathcal B\mathcal P}_L^G \vert \geq \vert {\mathcal B\mathcal P}_K \vert$ si et seulement si 
$\vert {\rm Im}(\theta)\vert \geq \vert {\rm Im}(\psi)\vert$.
\end{corollary}

\subsection{Caract\'erisation du cas $\vert {\rm Im}(\psi)\vert = p$ pour $L/K$ cyclique de degr\'e $p$} \label{cycl}
 Pour un cadre plus g\'en\'eral voir \cite{Ng2}, \S\,3 et \cite{Ja6}.
On \'ecarte le cas trivial ${\rm H}^1(G, \mu(L))=0$ o\`u $\vert{\rm Im}(\psi)\vert=1$ et
$\vert {\mathcal B\mathcal P}_L^G \vert = \vert {\mathcal B\mathcal P}_K \vert \,.\, \vert {\rm Im}(\theta)\vert$.
On suppose donc que $\mu(K) = \mu(L) \ne 1$ (i.e., ${\rm H}^1(G, \mu(L)) = \langle \zeta_1 \rangle$).

\smallskip
On a $\vert {\rm Im}(\psi)\vert = p$ si et seulement si $\nu = 0$ dans ($2'$)~; 
or l'image de $\zeta_1$ par $\nu$ est la classe de l'image diagonale de $\zeta_1 \pmod {W_L^{1-s} }$. 
Cette classe est alors nulle si et seulement si $\zeta_1 \in W_{L,v}^{1-s}$ pour tout $v \div p$ donc
 si et seulement si $j_{L/K}$ est {\it non injectif}  (Lemme \ref{kummer}).
D'apr\`es le Th\'eor\`eme \ref{thmf} et puisque $\mu(L)=\mu(K) \ne 1$, 
$j_{L/K}$ est non injectif si et seulement $L/K$ est $p$-ramifi\'ee et si pour toute place $v \div p$, 
non d\'ecompos\'ee dans $L$, on a $\mu(K_v) = \mu(L_w)^p$.

\smallskip
En r\'esum\'e, on a obtenu la condition n\'ecessaire et suffisante suivante~:

\begin{proposition} Soit $L/K$ cyclique de degr\'e $p$ v\'erifiant la conjecture de Leopoldt pour~$p$.
On a $\vert {\rm Im}(\psi)\vert = p$ si et seulement si $\mu(K) = \mu(L) \ne 1$,
$L/K$ est $p$-ramifi\'ee, et si pour toute place $v \div p$, non d\'ecompos\'ee dans $L$, on a $\mu(K_v) = \mu(L_w)^p$.
\end{proposition}

\begin{corollary} On suppose que l'extension $L/K$ est $p$-primitivement ramifi\'ee. 
Alors si $L/K$ est ramifi\'ee en au moins une place mod\'er\'ee ou s'il existe $v_0 \div p$ non 
d\'ecom\-pos\'ee telle que $\mu(K_{v_0}) = \mu(L_{w_0})$, alors on a $\vert {\rm Im}(\psi) \vert = 1$ et
$\vert {\mathcal B\mathcal P}_L^G \vert = \vert {\mathcal B\mathcal P}_K \vert \, \vert {\rm Im}(\theta) \vert$.
C'est aussi trivialement le cas si $L/K$ est cyclotomique globale.
\end{corollary}

Si $\vert {\rm Im}(\psi) \vert = p$, il suffit d'avoir $\vert{\rm Im}(\theta)\vert \geq p$
 pour que l'in\'egalit\'e $\vert {\mathcal B\mathcal P}_L^G \vert \geq \vert {\mathcal B\mathcal P}_K \vert$ 
soit v\'erifi\'e~; or ${\rm Im}(\theta)$ d\'epend non trivialement
du noyau de $\rho : {\rm H}^1(G, W_L/\mu(L)) \!\too {\rm H}^1(G, {\mathcal T}_L)$.

\subsection{Calcul du noyau de $\rho  : {}_N^{}(W_L/\mu(L)) \big /(W_L/\mu(L))^{1-s} \too  
{}_N^{}{\mathcal T}_L \big /{\mathcal T}_L^{1-s}$}\label{rho}
On reprend la suite exacte (1$'$) en la prolongeant comme suit~:

\medskip\noindent
$1 \to (W_L/\mu(L))^G \too {\mathcal T}_L^G \simeq {\mathcal T}_K 
\too {\mathcal B\mathcal P}_L^G \ds \mathop{\too}^{\theta}$

\smallskip
\hfill ${}_N^{}(W_L/\mu(L)) \big /(W_L/\mu(L))^{1-s} 
\ds\mathop{\too}^{\rho} {\rm H}^1(G, {\mathcal T}_L) =
{}_N^{}{\mathcal T}_L \big /{\mathcal T}_L^{1-s}.$

\medskip
 On est donc dans le cadre suivant pour $L/K$ cyclique $p$-ramifi\'ee (Th\'eor\`eme \ref{thmf})~:

\smallskip
(i) $\mu(K) = \mu(L) \ne 1$ ($L/K$ non globalement cyclotomique),

\smallskip
(ii) pour toute place $v \div p$ non d\'ecompos\'ee, on a 
$\mu(K_v) = \mu(L_w)^p$ ($L/K$ localement cyclotomique en $p$).

\begin{lemma} Sous les hypoth\`eses pr\'ec\'edentes, on a ${\rm H}^1(G, W_L/\mu(L)) \simeq \Z/p\Z$. Un g\'en\'e\-rateur de 
ce groupe est donn\'e par l'image d'un $\xi_{L,0}\in W_L$ tel que $\No_{L/K}(\xi_{L,0}) = \zeta_{L,0}$, o\`u $\zeta_{L,0}$ est un g\'en\'erateur de $\mu(K)$.
\end{lemma}  

\begin{proof}La suite exacte $1 \to \mu(L) \too W_L \too W_L/\mu(L) \to 1$ conduit \`a~:
$${\rm H}^1(G, W_L) \too {\rm H}^1(G, W_L/\mu(L))\mathop{\too}^{\nu} {\rm H}^2(G,\mu(L))\too {\rm H}^2(G, W_L). $$

\noindent
Or ${\rm H}^1(G, W_L)= \bigoplus_{v \div p}{\rm H}^1(G, W_{L,v})=0$
car c'est vrai dans le cas d\'ecompos\'e (lemme de Shapiro) et dans 
le cas non d\'ecompos\'e car $\mu(K_v) = \mu(L_w)^p$ (preuve du Th\'eor\`eme \ref{cns2}). 
Ensuite, ${\rm H}^2(G, W_L)=0$ puisque
$W_L$ est fini. D'o\`u ${\rm H}^1(G, W_L/\mu(L)) \simeq {\rm H}^2(G,\mu(L))
\simeq \Z/p\Z$.

\smallskip
Comme ${\rm H}^2(G, W_L)=0$, l'existence de $\xi_{L,0}$ est imm\'ediate 
et son image dans le groupe ${\rm H}^1(G, W_L/\mu(L))$ est d'ordre $p$. 
\end{proof}

A $\xi_{L,0} \in W_L$ (telle que $\No_{L/K}(\xi_{L,0}) =i_K(\zeta_{K,0})$), 
$\rho$ associe l'\'el\'ement $\tau_{L,0} \in {}_N^{}{\mathcal T}_L \big /{\mathcal T}_L^{1-s}$ ainsi d\'efini~:
on a $\xi_{L,0} = i_L(y)$, $y \in {\mathcal L}_1^\times$ avec $\No_{L/K}(y) = \zeta_{K,0} \, x_\infty$, 
$x_\infty \in {\mathcal K}_\infty^\times$~; 
alors l'image de $\tau_{L,0}$ dans ${\mathcal I}_L/{\mathcal P}_{L, \infty}$ est celle de l'id\'eal $(y)$.
On aura nullit\'e de $\rho$ si et seulement si $\tau_ {L,0}=\tau'{}^{1-s}_{\!\! L,0}$ pour un $\tau'_{L,0} \in {\mathcal T}_L$.

\smallskip
Il semble que tous les cas se rencontrent~; pour cela il faut 
pouvoir \'etudier les groupes de torsion ${\mathcal T}_k$
de fa\c con effective (pour un corps de nombres $k$), ce qui est possible en utilisant le logarithme 
${\rm Log}$ introduit dans \cite{Gr1}, III.2.2,
et qui conduit \`a la suite exacte~:
$$1 \to {\mathcal T}_k \tooo {\mathcal A}_k \simeq 
{\mathcal I}_k/{\mathcal P}_{k, \infty} \mathop{\tooo}^{{\rm Log}_k}  {\rm Log}_k({\mathcal I}_{k}) = \Z_p {\rm Log}_k(I_{k}) \to 0, $$
o\`u l'on rappelle que ($ {\rm log}_k$ \'etant le logarithme 
$p$-adique d'Iwasawa usuel)~:
$${\rm Log}_k \ : \ {\mathcal I}_k/{\mathcal P}_{k, \infty}\tooo 
\big( \hbox{$\bigoplus_{v \div p}$} k_v\big )\big /\Q_p  {\rm log}_k(E_k), $$

\noindent
est d\'efini pour tout id\'eal ${\mathfrak a}$ de $k$ \'etranger \`a $p$ par~: 
$${\rm Log}_k({\mathfrak a}) := \Frac{1}{m}{\rm log}_k(a) \pmod {\Q_p  {\rm log}_k(E_k)}, $$ 
o\`u $m$ est n'importe quel entier tel que ${\mathfrak a}^m$ est un id\'eal principal $(a)$.

\medskip
On a $\rho=0$ s'il existe un id\'eal ${\mathfrak b} \in  {\mathcal I}_L$ tel que
${\rm Log}_L( {\mathfrak b}) = 0$ et tel que $(y) = {\mathfrak b}^{1-s} \,(y_\infty)$.

\smallskip
En pratique ${\rm Log}_L({\mathcal P}_L)$ est canonique 
comme image de ${\rm log}_L(U_L)$ dans le $\Q_p$-espace 
$\big( \bigoplus_{w \div p} L_w\big ) \big / \Q_p {\rm log}_L(E_L)$, et
le $\Z_p$-module ${\rm Log}_L({\mathcal I}_L)$ est connu 
num\'eriquement d\`es que le groupe des classes l'est (au moyen 
d'id\'eaux g\'en\'erateurs ${\mathfrak b}_1$, \ldots, ${\mathfrak b}_r$ modulo le groupe des 
id\'eaux principaux de $L$), \`a savoir~:
$${\rm Log}_L({\mathcal I}_L)=\big\langle {\rm Log}_L({\mathfrak b}_1), \ldots, {\rm Log}_L({\mathfrak b}_r) \big\rangle_{\Z_p} 
\hbox{\Large +}  \  {\rm Log}_L(U_L). $$

\smallskip
Tout ceci d\'epend donc largement du groupe des unit\'es de $L$.
Il para\^it clair, en raison des valeurs des 
${\rm Log}_L({\mathfrak b}_i)$ et de ${\rm Log}_L((y))$ que, 
num\'eriquement, tout type d'exemple est possible comme le montre 
l'exp\'erience de ces calculs logarithmiques d'id\'eaux pour lesquels
il suffit de calculer modulo une puissance de $p$ convenable.

\smallskip
Comme  le quotient de Herbrand de ${\mathcal T}_L$ est nul on a 
$\vert  {}_N^{}{\mathcal T}_L \big /{\mathcal T}_L^{1-s} \vert  =\Frac{\vert  {\mathcal T}_K\vert }{\vert\No_{L/K}( {\mathcal T}_L) \vert}$.
Or la norme de ${\mathcal T}_L$ dans $L/K$ correspond, par la loi de r\'eciprocit\'e d'Artin, \`a la restriction
${\mathcal T}_L \too {\mathcal T}_K$ des automorphismes, et une condition suffisante pour que $\rho$ soit nulle est que
cette norme soit surjective. D'o\`u le r\'esultat suivant (sous la conjecture de Leopoldt)~:

\begin{theorem}\label{surj} Soit $L/K$ une extension cyclique de degr\'e $p$,
kummerienne, $p$-ramifi\'ee, et non cyclotomique 
(i.e., $\mu(K)=\mu(L) \ne 1$). On suppose que pour toute place $v \div p$,
non d\'ecompos\'ee dans $L/K$, on a $\mu(L_w)^p=\mu(K_v)$ (cf. Th\'eor\`eme \ref{thmf}). 

\smallskip\noindent
Alors une condition  suffisante pour que 
$\vert {\mathcal B\mathcal P}_L^G \vert \geq \vert {\mathcal 
B\mathcal P}_K \vert$ (i.e., $\vert {\rm Im}(\theta)\vert \geq \vert {\rm Im}(\psi)\vert = p$)
est que $L \subset \widetilde K$ et que
$H_K^{\rm pr}/\widetilde K$ et $\widetilde L / \widetilde K$ soient lin\'eairement disjointes sur $\widetilde K$.
\end{theorem}

La seconde condition peut se v\'erifier comme suit~: soit $F$ une extension de $K$ telle que 
$\widetilde K\, F=H_K^{\rm pr}$~; alors $FL/L$ et $\widetilde L / L$ doivent \^etre lin\'eairement 
disjointes sur $L$.

\section{Conclusion}
Le cas trivial o\`u $\mu(K)=1$ (Th\'eor\`eme \ref{trivial}) permet de construire les $p$-tours  d\'efinies par 
Seo dans \cite{Seo6}, lesquelles sont infinies dans le cas o\`u ${\mathcal B\mathcal P}_K \ne 1$
(sous la conjecture de Leopoldt pour $p$ et toute $p$-extension de $K$), mais dans les autres cas se pose le 
probl\`eme de la propagation des hypoth\`eses faites au premier \'etage.

\smallskip
Le Th\'eor\`eme \ref{surj} montre des liens avec le plongement
kummerien dans les $\Z_p$-extensions (cf. \S\,\ref{histoire}\,(i)),
dont on connait l'effectivit\'e num\'erique.

\smallskip
Lorsque le $p$-groupe des classes de $K$ est trivial (ou bien lorsque $H_K \subset \widetilde K$) on a~:
$${\mathcal B\mathcal P}_K \simeq {\rm tor}_{\Z_p} \big ({\rm log}_K(U_K) / \Z_p\,{\rm log}_K (E_K)\big ), $$ 
qui d\'epend simplement des propri\'et\'es congruentielles du groupe 
des unit\'es de $K$ et grosso modo du r\'egulateur $p$-adique normalis\'e de $K$ (\cite{Gr7}, D\'efinition 2.3)~;
dans cet article nous avons conjectur\'e, pour tout $p$ assez grand, 
la $p$-rationalit\'e d'un corps de nombres $K$ (voir les propri\'et\'es de cette notion dans 
\cite{Gr1}, \cite{Gr5}, \cite{GrJ},\cite{Ja4}, \cite{JaNg}, \cite{Mo}, \cite{MoNg1})~; ceci entra\^inerait la nullit\'e de 
${\mathcal B\mathcal P}_K$ pour tout $p$ assez grand.

\end{document}